\documentclass{amsart}
\usepackage{amsfonts}
\usepackage{amsmath}
\usepackage{enumerate}
\usepackage{amsthm}
\usepackage{hyperref}
\usepackage{cite}

\begin{document}

\newtheorem{definition}{Definition}[section]
\newtheorem{theorem}[definition]{Theorem}
\newtheorem{proposition}[definition]{Proposition}
\newtheorem{remark}[definition]{Remark}
\newtheorem{lemma}[definition]{Lemma}
\newtheorem{corollary}[definition]{Corollary}
\newtheorem{example}[definition]{Example}

\numberwithin{equation}{section}

\title[A note on the Chern-Yamabe flow...]{A note on the Chern-Yamabe flow on compact Hermitian manifolds with negative Gauduchon degree}
\author{Weike Yu}

\thanks{The author is supported by NSFC (No. 12501075), Basic Research Program of Jiangsu (No. BK20250644), Natural Science Foundation of the Jiangsu Higher Education Institutions of China (No. 25KJB110008), the start-up research funds from Nanjing Normal University with account No. 184080H201B160.}
\date{}

\begin{abstract}
In this note, we prove that on a compact Hermitian manifold with negative Gauduchon degree, the Chern-Yamabe flow introduced in our recent work \cite{[Yu1]} converges to the solution of prescribed Chern scalar curvature problem, provided that the candidate curvature function is nonpositive and not identically zero. 
\end{abstract}
\maketitle
\section{Introduction}
Suppose that $(M^n, J, h)$ is a compact complex manifold of complex dimension $n\geq2$ equipped with a complex structure $J$ and a Hermitian metric $h$. The fundamental form $\omega$ of $h$ is defined by $\omega(\cdot, \cdot)=h(J\cdot, \cdot)$.  In this paper, we shall identify the Hermitian metric $h$ with its corresponding fundamental form $\omega$. For a Hermitian metric $\omega$, there exists a $1$-form $\theta$ on $M$, which is called the Lee form or torsion $1$-form, such that
\begin{align}
d\omega^{n-1}=\theta\wedge \omega^{n-1}.
\end{align}
A Hermitian metric $\omega$ is said to be Gauduchon if $d^*\theta=0$, where $d^*$ denotes the adjoint of $d$ with respect to $\omega$. In particular, it is said to be balanced if $\theta=0$. 

On a Hermitian manifold $(M^n, \omega)$, there is a unique linear connection $\nabla^{Ch}$ (called the Chern connection) preserving both the Hermitian metric and the complex structure, whose torsion $T^{Ch}$ has vanishing $(1,1)$-part everywhere. The scalar curvature with respect to $\nabla^{Ch}$ (referred to as the Chern scalar curvature) is given by
\begin{align}
S^{Ch}(\omega)=\text{tr}_{\omega}Ric^{(1)}(\omega)=\text{tr}_{\omega}\sqrt{-1}\bar{\partial}\partial \log{\omega^n}, 
\end{align}
where $Ric^{(1)}(\omega)$ denotes the first Chern Ricci curvature of $(M, \omega)$ and $\omega^n$ denotes the volume form of $\omega$.

On a compact Hermitian manifold $(M^n, \omega)$, the following problem is a Hermitian analogue of prescribing scalar curvature: 
\begin{center}
\begin{minipage}{8cm}
Given a smooth real-valued function $g$ on $M$, does there exist a Hermitian metric $\hat{\omega}$ conformal to $\omega$, i.e., $\hat{\omega}=e^{\frac{2}{n}u}\omega$, such that its Chern scalar curvature $S^{Ch}(\hat{\omega})=g$? 
\end{minipage}
\end{center}
According to \cite{[ACS]}, the above problem is equivalent to solving the following partial differential equation:
\begin{align}
-\Delta^{Ch}_\omega u+S^{Ch}(\omega)=ge^{\frac{2}{n}u},\label{1.3}
\end{align}
where $\Delta^{Ch}_\omega$ denotes the Chern Laplacian with respect to $\omega$. If $g$ is a constant, the above problem is referred to as the Chern-Yamabe problem, which was first proposed by Angella-Calamai-Spotti in \cite{[ACS]}; they proved that it is solvable on any compact Hermitian manifold with Gauduchon degree $\Gamma(\{\omega\})\leq 0$ (see \eqref{2.6.}).  If $g$ is a general smooth function, this problem on compact (almost) Hermitian manifolds has been studied by several authors (cf. \cite{[LM], [CZ], [Ho], [Fus], [Bar], [Yu1], [WZ1], [WZ2], [LU], [LZZ]}).

In \cite{[ACS]}, the authors proposed using the following flow (called the Chern-Yamabe flow) to tackle the above Chern-Yamabe problem:
\begin{equation}
\left\{
\begin{aligned}
&\frac{\partial u}{\partial t}=\Delta^{Ch}_\omega u-S(\omega)+\lambda e^{\frac{2}{n}u},\\
&u(x,0)=0,
\end{aligned}
\right.
\end{equation}
where $\lambda$ is a constant. Later, Lejmi-Maalaoui \cite{[LM]} proved that when the Gauduchon degree is negative, the above flow converges to a solution of the Chern-Yamabe problem, which recovered the result of Angella-Calamai-Spotti. In \cite{[CZ]}, Calamai-Zou introduced a different flow to study the Chern-Yamabe problem. Using geometric flows related to the Calamai-Zou flow, Ho \cite{[Ho]} studied the prescribed Chern scalar curvature problem on compact balanced Hermitian manifolds with negative Chern scalar curvature. 

In \cite{[Yu1]}, we studied the following geometric flow:
\begin{equation}
\left\{
\begin{aligned}
&\frac{\partial u}{\partial t}=\Delta^{Ch}_\eta u-s_0+ge^{\frac{2}{n}u},\\
&u(x,0)=0,
\end{aligned}
\right.
\label{2.1}
\end{equation}
where $\eta$ is the unique Gauduchon representative in the conformal class $\{\omega\}$ with volume $1$ (see Theorem \ref{theorem2.2}), $s_0$ is the Chern scalar curvature with respect to $\eta$, and $g$ is a smooth function on $M$ with $g\leq 0$ and  $g\not\equiv0$. Indeed, we proved that

 \begin{theorem}[\cite{[Yu1]}]\label{thm1.1}
Let $(M^n,\omega)$ be a compact Hermitian manifold with complex dimension $n\geq 2$ and Gauduchon degree $\Gamma (\{\omega\})<0$, and assume that $g\in C^\infty(M)$ is nonpositive and not identically zero. Then the flow \eqref{2.1} has a unique smooth solution $u$ defined on $M\times [0,+\infty)$. Furthermore, suppose either
\begin{enumerate}
\item $g$ is negative on $M$,
\end{enumerate}
or
\begin{enumerate}[\ \ \ \ (2)]
\item $g$ is nonpositive and not identically zero on $M$, and $\{\omega\}$ contains a balanced metric,
\end{enumerate}
then there exists a sequence $t_i\rightarrow +\infty$ such that $u(\cdot,t_i)$ converges to $u_\infty$ in $C^{1,\alpha}(M)$, where $\alpha\in (0,1)$, and $u_\infty$ is the unique smooth solution of 
\begin{align}\label{1.6.}
-\Delta^{Ch}_\eta u_\infty+s_0=ge^{\frac{2}{n}u_\infty}.
\end{align}
\end{theorem}

In this paper, we remove the assumption that $\{\omega\}$ contains a balanced metric in Theorem \ref{thm1.1} (2) and upgrade subsequential convergence to convergence of the full flow. More precisely, we obtain the following theorem.
 \begin{theorem}\label{theorem1.1}
Let $(M^n,\omega)$ be a compact Hermitian manifold with complex dimension $n\geq 2$ and Gauduchon degree $\Gamma (\{\omega\})<0$, and assume that $g\in C^\infty(M)$ is nonpositive and not identically zero. Then the flow \eqref{2.1} has a unique smooth solution $u$ defined on $M\times [0,+\infty)$. Furthermore, $u(\cdot,t)$ converges to $u_\infty$ in $C^{1,\alpha}(M)$ as $t\rightarrow +\infty$, where $\alpha\in (0,1)$, and $u_\infty$ is the unique smooth solution of 
\begin{align}\label{1.7+}
-\Delta^{Ch}_\eta u_\infty+s_0=ge^{\frac{2}{n}u_\infty}.
\end{align}
\end{theorem}

\begin{remark}
By making a small modification to the argument in Section 3, we may see that Theorem \ref{theorem1.1} still holds for the flow \eqref{2.1} with the initial data $u(x, 0)=u_0\in C^\infty(M)$.
\end{remark}

Let $PC(\omega)$ denote the set of $C^\infty(M)$ functions which are Chern scalar curvatures of all $\hat\omega \in \{\omega\}$. In other words, $PC(\omega)$ is the set of smooth functions $g$ for which one can find a solution of \eqref{1.3}. According to Theorem \ref{theorem1.1}, we have the following corollary, which recovers \cite[Theorem 2.5]{[Fus]}. 
\begin{corollary}\label{corollary1.3}
Let $(M^n,\omega)$ be a compact Hermitian manifold with complex dimension $n\geq 2$ and Gauduchon degree $\Gamma (\{\omega\})<0$. Suppose that $g\in C^\infty(M)$ is nonpositive and not identically zero. Then $g\in PC(\omega)$.
\end{corollary}

\begin{remark}
Using a different geometric flow, Ho \cite{[Ho]} proved that if $(M^n,\omega)$ is a compact balanced Hermitian manifold with negative Chern scalar curvature, then $\{g\in C^\infty(M): g<0\}\subset PC(\omega)$.
\end{remark}

\textbf{Acknowledgements.} The author would like to thank Prof. Yuxin Dong and Prof. Xi Zhang for their continued support and encouragement. 

\section{Preliminaries}\label{sec2}
Let $(M^n, J, h)$ be a Hermitian manifold with $\dim_{\mathbb{C}} M=n$ and corresponding fundamental form $\omega(\cdot,\cdot)=h(J\cdot, \cdot)$. Let $TM^{\mathbb{C}}=TM\otimes \mathbb{C}$ be the complexified tangent space of $M$, and extend $J$ and $h$ from $TM$ to $TM^{\mathbb{C}}$ by $\mathbb{C}$-linearity. Then we have the following decomposition
\begin{align}
TM^{\mathbb{C}}=T^{1,0}M\oplus T^{0,1}M,
\end{align}
where $T^{1,0}M$ and $T^{0,1}M$ are the eigenspaces of complex structure $J$ corresponding to the eigenvalues $\sqrt{-1}$ and $-\sqrt{-1}$, respectively. Note that every $m$-form can also be decomposed into $(p, q)$-forms for each $p, q \geq 0$ with $p + q = m$ by extending the complex structure $J$ to forms.

On a Hermitian manifold $(M^n, J, h)$, there exists a unique affine connection $\nabla^{Ch}$ preserving both the Hermitian metric $h$ and the complex structure $J$, that is, $\nabla^{Ch}h=0$, $\nabla^{Ch}J=0$, whose torsion $T^{Ch}(X,Y)=\nabla^{Ch}_XY-\nabla^{Ch}_YX-[X,Y]$ has vanishing $(1,1)$-part everywhere. Note that throughout this paper, we will not distinguish between the Hermitian metric $h$ and its corresponding fundamental form $\omega$.

On a Hermitian manifold $(M^n, \omega)$ with the Chern connection $\nabla^{Ch}$, it is well known that the corresponding Chern scalar curvature is given by
\begin{align}
S^{Ch}(\omega)=\text{tr}_{\omega}Ric^{(1)}(\omega)=\text{tr}_{\omega}\sqrt{-1}\bar{\partial}\partial \log{\omega^n},
\end{align}
where $Ric^{(1)}(\omega)$ is the first Chern Ricci curvature.

Analogous to the Laplace operator, there is a canonical elliptic operator on Hermitian manifolds, called the Chern Laplace operator. More precisely, for any $u\in C^\infty(M)$, the Chern Laplacian $\Delta^{Ch}_\omega$ of $u$ is defined by
\begin{align}
\Delta^{Ch}_\omega u=-2\sqrt{-1} tr_{\omega}\overline{\partial}\partial u,
\end{align}
or, in local holomorphic coordinates $\{z^i\}$ such that $\omega\overset{\text{loc}}{=}\sqrt{-1}\sum_{i, j}h_{i\bar{j}}dz^i\wedge dz^{\bar{j}}$, as
\begin{align}
\Delta^{Ch}_\omega u\overset{\text{loc}}{=}2\sum_{i, j}h^{i\bar{j}}\frac{\partial^2u}{\partial z^i\partial z^{\bar{j}}},
\end{align}
where $(h^{i\bar{j}})$ is the inverse matrix of $(h_{i\bar{j}})$.

\begin{lemma}[cf. \cite{[Gau]}]\label{lemma2.1}
On a compact Hermitian manifold $(M^n,\omega)$, we have
\begin{align}
-\Delta^{Ch}_\omega u=-\Delta_d u+(du,\theta)_{\omega},
\end{align}
where $\Delta_d=-d^*d$ is the Hodge-de Rham Laplacian, $d^*$ denotes the adjoint of $d$ with respect to $\omega$, and $\theta$ is the Lee form or torsion $1$-form given by $d\omega^{n-1}=\theta\wedge \omega^{n-1}$, $(\cdot, \cdot)_\omega$ denotes the inner product on $1$-form induced by $\omega$. In particular, when $\omega$ is balanced (i.e., $\theta\equiv0$), the Chern Laplacian and the Hodge-de Rham Laplacian on smooth functions coincide.
\end{lemma}
Let 
\begin{align}
\{\omega\}=\{e^{\frac{2}{n}u}\omega\ |\ u\in C^\infty(M)\}
\end{align}
denote the conformal class of the Hermitian metric $\omega$. In \cite{[Gau1]}, Gauduchon proved the following interesting result.
\begin{theorem}\label{theorem2.2}
Let $(M, \omega)$ be a compact Hermitian manifold with $\dim_{\mathbb{C}} M\geq 2$, then there exists a unique Gauduchon metric $\eta\in \{\omega\}$ (i.e., $d^*\theta=0$) with volume $1$.
\end{theorem}
In terms of the above theorem, one defines an invariant $\Gamma(\{\omega\})$ of the conformal class $\{\omega\}$ as follows:
\begin{align}\label{2.6.}
\Gamma(\{\omega\})=\frac{1}{(n-1)!}\int_M c^{BC}_1(K^{-1}_M)\wedge\eta^{n-1}=\int_M S^{Ch}(\eta)d\mu_\eta,
\end{align}
where $c^{BC}_1(K^{-1}_M)$ is the first Bott-Chern class of anti-canonical line bundle $K^{-1}_M$, and $d\mu_\eta$ denotes the volume form of the Gauduchon metric $\eta$.

On a compact Hermitian manifold $(M^n, \omega)$, we consider the conformal change $\widetilde{\omega}=e^{\frac{2}{n}u}\omega$. According to \cite{[Gau]}, the Chern scalar curvatures of $\widetilde{\omega}$ and $\omega$ have the following relationship:
\begin{align}\label{2.7}
-\Delta^{Ch}_\omega u+S^{Ch}(\omega)=S^{Ch}(\widetilde{\omega})e^{\frac{2}{n}u}.
\end{align}
In order to attack the prescribed Chern scalar curvature problem, it is useful to view equation \eqref{2.7} as an Euler-Lagrange equation of some functional. Unfortunately, in general, such a functional does not always exist. More precisely, from \cite[Prop. 5.3]{[ACS]} and \cite[Prop. 2.12]{[Fus]}, we know the following
\begin{proposition}
Equation \eqref{2.7} can be viewed as the Euler-Lagrange equation for the standard $L^2$ pairing if and only if $\omega$ is balanced.
\end{proposition}

\section{Proof of Main theorem}

\begin{lemma}\label{lemma3.1}
The flow \eqref{2.1} has a unique smooth solution $u$ defined on $M\times [0,+\infty)$. Furthermore, there exists a constant $C_1>0$ independent of $t\in [0,+\infty)$ such that 
 \begin{align}\label{3.1-}
 \|u\|_{C^0(M)}<C_1, \quad \left\|\frac{\partial u}{\partial t}\right\|_{C^0(M)}<C_1,
 \end{align}
 and thus
 \begin{align}\label{3.2-}
 \|u\|_{C^{1, \alpha}(M)}<C_2,
 \end{align}
 for any $\alpha\in (0,1)$, and $C_2$ is a constant independent of $t\in [0,+\infty)$. Here and throughout,  $C_i\ (i=1, 2, 3, \dots)$ denotes a generic constant independent of $t\in [0,+\infty)$.
 \end{lemma}
 \proof The existence of long-time solution of \eqref{2.1} and the estimates \eqref{3.1-} have been proved in \cite{[Yu1]}. Next we prove \eqref{3.2-}. From \eqref{2.1} and \eqref{3.1-}, we obtain
 \begin{equation}\label{3.3-}
  \begin{aligned}
\left\| \Delta^{Ch}_\eta u\right\|_{C^0(M)}&=\left\|\frac{\partial u}{\partial t}+s_0-ge^{\frac{2}{n}u}\right\|_{C^0(M)}\\
&\leq \left\|\frac{\partial u}{\partial t}\right\|_{C^0(M)}+\left\|s_0\right\|_{C^0(M)}+\sup_M|g|\cdot e^{\frac{2}{n}\|u\|_{C^0(M)}}\\
&\leq C_1+\left\|s_0\right\|_{C^0(M)}+\sup_M|g|\cdot e^{\frac{2}{n}C_1}=:C_3.
\end{aligned}
\end{equation}
Using the $L^p$-estimate for elliptic equations and \eqref{3.3-} yields for any $p>1$
\begin{align}
\|u\|_{W^{2,p}(M)}\leq C_4\left(\| \Delta^{Ch}_\eta u\|_{L^p(M)}+\|u\|_{L^p(M)} \right)\leq C_4(C_3+C_1),
\end{align}
that is, $\|u\|_{W^{2,p}(M)}$ is bounded uniformly with respect to $t\in [0, +\infty)$ for any $p>1$. Here $W^{2, p}(M)$ denotes the Sobolev space on $M$ with respect to $\eta$. Therefore, for any $\alpha\in (0,1)$, taking $p>\frac{2n}{1-\alpha}$ and using the Sobolev embedding theorem $W^{2,p}(M)\subset C^{1,\alpha}(M)$, we obtain
\begin{align}
\|u\|_{C^{1, \alpha}(M)}\leq C_5\|u\|_{W^{2,p}(M)}\leq C_5C_4(C_3+C_1)=:C_2,
\end{align}
which finishing the proof of \eqref{3.2-}.
\qed

\proof[Proof of Theorem \ref{theorem1.1}]
Set 
\begin{align}\label{3.6..}
f=\partial_t u,
\end{align}
where $\partial_t=\frac{\partial}{\partial t}$. From \eqref{2.1} and  $g\leq 0$, we deduce that
\begin{equation}\label{3.1}
\begin{aligned}
\left(\partial_t-\Delta^{Ch}_\eta\right)f^2=\frac{4}{n}ge^{\frac{2}{n}u}f^2-2|\nabla f|^2\leq -2|\nabla f|^2.
\end{aligned}
\end{equation}
Since $\eta$ is Gauduchon (i.e., $d^*\theta=0$), it follows from Lemma \ref{lemma2.1} that
\begin{align}\label{3.2}
\int_M \Delta^{Ch}_\eta f^2d\mu_\eta=\int_M\Delta_d f^2d\mu_\eta-\int_Mf^2d^*\theta d\mu_\eta=0.
\end{align}
Integrating both sides of \eqref{3.1} over $M$ with respect to $\eta$ and using \eqref{3.2} yields
\begin{equation}\label{3.3}
\begin{aligned}
\frac{d}{dt}\int_M f^2d\mu_\eta=\int_M \partial_tf^2d\mu_\eta\leq -2\int_M|\nabla f|^2d\mu_\eta.
\end{aligned}
\end{equation}
Applying the Poincar\'e inequality on the compact manifold $(M, \eta)$: there exists a constant $C_P>0$ independent of $t$ such that
\begin{align}\label{3.4}
\int_M|f-\overline{f}|^2d\mu_\eta\leq C_P\int_M|\nabla f|^2d\mu_\eta,
\end{align}
where $\overline{f}=\int_M fd\mu_\eta$.
Combining \eqref{3.3} and \eqref{3.4}, we get
\begin{align}
\frac{d}{dt}\int_M f^2d\mu_\eta\leq -\frac{2}{C_P}\int_M|f-\overline{f}|^2d\mu_\eta.
\end{align}
For any $T>0$, integrating the above inequality over $[0, T]$ gives
\begin{equation}
\begin{aligned}
\int^T_0\int_M|f-\overline{f}|^2d\mu_\eta dt&\leq \frac{C_P}{2}\left(\|f\|_{L^2(M)}^2(0)-\|f\|_{L^2(M)}^2(T) \right)\\
&\leq \frac{C_P}{2}\|f\|_{L^2(M)}^2(0)<+\infty,
\end{aligned}
\end{equation}
which implies
\begin{align}\label{3.7.}
\int^{+\infty}_0\int_M|f-\overline{f}|^2d\mu_\eta dt<+\infty.
\end{align}

Set 
\begin{equation}
\begin{aligned}
V(t):&=\int_M|f-\overline{f}|^2d\mu_\eta\\
&=\int_M\left(f^2-2f\overline{f}+\overline{f}^2\right)d\mu_\eta\\
&=\int_M f^2d\mu_\eta-\overline{f}^2.
\end{aligned}
\end{equation}
Here we have used $\int_M d\mu_\eta=1$. Thus,
\begin{equation}\label{3.15+}
\begin{aligned}
V'(t)=\frac{d}{dt}\int_M f^2d\mu_\eta-2\overline{f}\int_M\partial_t fd\mu_\eta.
\end{aligned}
\end{equation}
From \eqref{2.1}, it follows that
\begin{align}\label{3.16+}
\partial_tf=\Delta^{Ch}_\eta f+\frac{2}{n}ge^{\frac{2}{n}u}f.
\end{align}
Substituting \eqref{3.16+} into \eqref{3.15+} and using \eqref{3.3} and \eqref{3.1-} yields
\begin{equation}\label{3.17+}
\begin{aligned}
V'(t)&\leq 0-2\overline{f}\int_M\left(\Delta^{Ch}_\eta f+\frac{2}{n}ge^{\frac{2}{n}u}f\right)d\mu_\eta\\
&=-2\overline{f}\int_M\frac{2}{n}ge^{\frac{2}{n}u}fd\mu_\eta\\
&\leq \frac{4}{n}|\overline{f}|\int_M|g|e^{\frac{2}{n}u}|f|d\mu_\eta\\
&\leq \frac{4}{n}C_1^2\sup_M|g|e^{\frac{2}{n}C_1}=:C_6,
\end{aligned}
\end{equation}
where $C_6>0$ is a constant independent of $t$. Now we claim that 
\begin{align}\label{3.9}
\lim_{t\rightarrow +\infty}V(t)=0.
\end{align}
If not, then since $V(t)\geq 0$, there exist a constant $\epsilon_0>0$ and a sequence $t_i\rightarrow +\infty$ such that
\begin{align}
V(t_i)\geq \epsilon_0,\quad \forall\ i.
\end{align}
By \eqref{3.17+}, we obtain that for any $s\in [t_i-\frac{\epsilon_0}{2C_6}, t_i]$,
\begin{align}
V(t_i)-V(s)=\int^{t_i}_sV'(t)dt\leq C_6(t_i-s),
\end{align}
and then
\begin{align}
V(s)\geq V(t_i)-C_6(t_i-s)\geq \epsilon_0-C_6\frac{\epsilon_0}{2C_6}=\frac{\epsilon_0}{2}.
\end{align}
Since $t_i\rightarrow +\infty$, we may choose a subsequence, still denoted by $t_i$, such that the intervals
\begin{align}
I_i=[t_i-\frac{\epsilon_0}{2C_6}, t_i]
\end{align}
are pairwise disjoint. Hence,
\begin{align}
\int^\infty_0V(t)dt\geq\sum_{i=1}^\infty\int_{I_i}V(t)dt\geq \sum_{i=1}^\infty\frac{\epsilon_0}{2}\frac{\epsilon_0}{2C_6}=+\infty,
\end{align}
which contradicts \eqref{3.7.}. Therefore, we finish the proof of \eqref{3.9}:
\begin{align}
\lim_{t\rightarrow +\infty}\int_M|f-\overline{f}|^2d\mu_\eta=0.
\end{align}

Next we will prove
\begin{align}\label{3.12..}
\lim_{t\rightarrow +\infty}\|f(\cdot, t)\|_{L^2(M)}=0.
\end{align}
To this end, it suffices to prove
\begin{align}\label{3.11}
\lim_{t\rightarrow +\infty}|\overline{f}|(t)=0,
\end{align}
because of \eqref{3.9} and  
\begin{equation}
\begin{aligned}
\|f(\cdot, t)\|_{L^2(M)}&\leq \|(f-\overline{f})(\cdot, t)\|_{L^2(M)}+\|\overline{f}(t)\|_{L^2(M)}\\
&= \|(f-\overline{f})(\cdot, t)\|_{L^2(M)}+|\overline{f}|(t).
\end{aligned}
\end{equation}
Now we start to prove \eqref{3.11}. By \eqref{2.1}, we get
\begin{equation}\label{3.26...}
\begin{aligned}
\overline{f}(t)&=\int_M\partial_t ud\mu_\eta\\
&=\int_M\Delta^{Ch}_\eta ud\mu_\eta-\int_Ms_0d\mu_\eta+\int_Mge^{\frac{2}{n}u}d\mu_\eta\\
&=0-\int_Ms_0d\mu_\eta+\int_Mge^{\frac{2}{n}u}d\mu_\eta\\
&=-\Gamma(\{\omega\})+\int_Mge^{\frac{2}{n}u}d\mu_\eta.
\end{aligned}
\end{equation}
From \eqref{3.26...} and $g\leq (\not\equiv) 0$, it follows that
\begin{align}\label{3.13}
\overline{f}(t)\leq -\Gamma(\{\omega\})-\delta,
\end{align}
where $\delta=-(\int_Mgd\mu_\eta) e^{-\frac{2}{n}C_1}>0$. Moreover, according to \eqref{3.26...} and \eqref{3.6..}, we have
\begin{equation}\label{3.14}
\begin{aligned}
\overline{f}'(t)&=0+\frac{d}{dt}\int_Mge^{\frac{2}{n}u}d\mu_\eta\\
&=\frac{2}{n}\int_Mge^{\frac{2}{n}u}\partial_t ud\mu_\eta\\
&=\frac{2}{n}\int_Mge^{\frac{2}{n}u}(f-\overline{f})d\mu_\eta+\frac{2}{n}\int_Mge^{\frac{2}{n}u}d\mu_\eta\cdot \overline{f}\\
&=\frac{2}{n}\int_Mge^{\frac{2}{n}u}(f-\overline{f})d\mu_\eta+\frac{2}{n}\overline{f}(\overline{f}+\Gamma(\{\omega\})).
\end{aligned}
\end{equation}
Let
\begin{align}
R(t)=\frac{2}{n}\int_Mge^{\frac{2}{n}u}(f-\overline{f})d\mu_\eta.
\end{align}
By \eqref{3.9} and \eqref{3.1-}, we obtain
\begin{align}\label{3.19}
|R(t)|\leq \frac{2}{n}\left(\int_Mg^2e^{\frac{4}{n}C_1}d\mu_\eta\right)^{\frac{1}{2}} \|f-\overline{f}\|_{L^2(M)}(t)\rightarrow 0,
\end{align}
as $t\rightarrow +\infty$. Define
\begin{align}
W_\epsilon(t)=\sqrt{|\overline{f}(t)|^2+\epsilon}, \quad \forall \epsilon>0.
\end{align}
Using \eqref{3.14} and \eqref{3.13}, we have
\begin{equation}\label{3.21}
\begin{aligned}
\frac{d}{dt}W_\epsilon(t)&=\frac{\overline{f}(t)\overline{f}'(t)}{\sqrt{(\overline{f}(t))^2+\epsilon}}\\
&=\frac{R(t)\overline{f}+\frac{2}{n}\overline{f}^2(\overline{f}+\Gamma(\{\omega\}))}{\sqrt{(\overline{f}(t))^2+\epsilon}}\\
&\leq \frac{|R(t)|\sqrt{|\overline{f}|^2+\epsilon}-\frac{2}{n}\delta(\overline{f}^2+\epsilon)+\frac{2}{n}\delta\epsilon}{\sqrt{(\overline{f}(t))^2+\epsilon}}\\
&\leq -\frac{2}{n}\delta W_\epsilon(t)+|R(t)|+\frac{2}{n}\delta\epsilon^{\frac{1}{2}}.
\end{aligned}
\end{equation}
From \eqref{3.19}, it follows that for any $\epsilon'>0$ there exists a constant $T_{\epsilon'}>0$ such that for any $t\geq T_{\epsilon'}$ we have
\begin{align}\label{3.22}
|R(t)|<\frac{2}{n}\delta\epsilon'.
\end{align}
Then from \eqref{3.21} and \eqref{3.22}, we obtain 
\begin{align}
\frac{d}{dt}\left(W_\epsilon(t)-\epsilon'-\epsilon^{\frac{1}{2}} \right)\leq -\frac{2}{n}\delta \left(W_\epsilon(t)-\epsilon'-\epsilon^{\frac{1}{2}} \right),\quad \forall\ t\geq T_{\epsilon'}.
\end{align}
which implies
\begin{align}
\frac{d}{dt}\left(e^{\frac{2}{n}\delta t}\left(W_\epsilon(t)-\epsilon'-\epsilon^{\frac{1}{2}} \right) \right)\leq 0, \quad \forall\ t\geq T_{\epsilon'}.
\end{align}
Thus,
\begin{align}\label{3.36--}
W_\epsilon(t)-\epsilon'-\epsilon^{\frac{1}{2}}\leq e^{-\frac{2}{n}\delta (t-T_{\epsilon'})}\left(W_\epsilon(T_{\epsilon'})-\epsilon'-\epsilon^{\frac{1}{2}} \right),  \quad \forall\ t\geq T_{\epsilon'}.
\end{align}
Letting $t\rightarrow +\infty$ in \eqref{3.36--}, we obtain
\begin{align}
\varlimsup_{t\rightarrow +\infty}|\overline{f}|(t)\leq \varlimsup_{t\rightarrow +\infty} \sqrt{(\overline{f}(t))^2+\epsilon}=\varlimsup_{t\rightarrow +\infty}W_\epsilon(t)\leq \epsilon'+\epsilon^{\frac{1}{2}}.
\end{align}
Letting $\epsilon, \epsilon'\rightarrow 0+$ in the above inequality, we have
\begin{align}
\lim_{t\rightarrow +\infty}|\overline{f}|(t)=0,
\end{align}
which finishing the proof of \eqref{3.11}. Therefore, we have proved \eqref{3.12..}, that is,
\begin{align}\label{3.39}
\lim_{t\rightarrow +\infty}\left\|\partial_t u\right\|_{L^2(M)}=0.
\end{align}

Finally, we will prove that there exists a function $u_\infty\in C^{1, \alpha}(M)$ such that
\begin{align}\label{3.42}
\lim_{t\rightarrow +\infty}\|u(\cdot, t)-u_\infty(\cdot)\|_{C^{1, \alpha}(M)}=0
\end{align}
for any $\alpha\in (0,1)$. Indeed, for any $\alpha\in (0, 1)$, we first choose $\beta\in (\alpha, 1)$. By \eqref{3.2-}, we have
\begin{align}\label{3.43++}
\|u\|_{C^{1, \beta}(M)}<C_2.
\end{align}
By the compact embedding theorem $C^{1, \beta}(M)\subset\subset C^{1,\alpha}(M)$, there exist a function $u_\infty\in C^{1, \alpha}(M)$ and a sequence $t_i\rightarrow +\infty$ such that
\begin{align}\label{3.47}
u(\cdot,t_i)\rightarrow u_\infty \text{\ in\ } C^{1,\alpha}(M).
\end{align}
Clearly, \eqref{3.39} and \eqref{3.47} imply $u_\infty$ is a weak solution of \eqref{1.7+}. By the regularity results for elliptic equations, we conclude that $u_\infty$ is a smooth solution of \eqref{1.7+}. Next we prove
\begin{align}\label{3.42+}
\lim_{t\rightarrow +\infty}\|u(\cdot, t)-u_\infty(\cdot)\|_{C^{1, \alpha}(M)}=0.
\end{align}
If \eqref{3.42+} does not hold, then there exist a constant $\epsilon_0>0$ and a sequence $t_k\rightarrow +\infty$ such that
\begin{align}\label{3.46--}
\|u(\cdot, t_k)-u_\infty(\cdot)\|_{C^{1, \alpha}(M)}\geq \epsilon_0, \quad \forall\ k\geq 1.
\end{align}
According to \eqref{3.43++} and the compact embedding theorem $C^{1, \beta}(M)\subset\subset C^{1,\alpha}(M)$, there exist a subsequence $t_{k_j}\rightarrow +\infty$ and a function $v\in C^{1, \alpha}(M)$ such that
\begin{align}\label{3.47--}
u(\cdot,t_{k_j})\rightarrow v \text{\ in\ } C^{1,\alpha}(M),
\end{align}
where $v$ is also a smooth solution of \eqref{1.7+} by the same argument as above. Set $w=v-u_\infty$. From \eqref{1.7+} and $g\leq (\not\equiv) 0$, we deduce that
\begin{align}\label{3.48++}
-\Delta^{Ch}_\eta w=g\left(e^{\frac{2}{n}v}-e^{\frac{2}{n}u_\infty}\right)=a(x)w,
\end{align}
where
\begin{align}
a(x)=\frac{2}{n}g(x)\int^1_0e^{\frac{2}{n}(sv+(1-s)u_\infty)}ds\leq (\not\equiv) 0.
\end{align}
Multiplying \eqref{3.48++} by $w$ and integrating over $M$ with respect to $\eta$ yields
\begin{equation}
\begin{aligned}
0\geq \int_M a w^2 d\mu_\eta&=-\int_M w\Delta_d wd\mu_\eta+\frac{1}{2}\int_M (dw^2, \theta)_\eta d\mu_\eta\\
&=\int_M|\nabla w|^2d\mu_\eta+\frac{1}{2}\int_M (w^2, d^*\theta)_\eta d\mu_\eta\\
&=\int_M|\nabla w|^2d\mu_\eta,
\end{aligned}
\end{equation}
and hence $w\equiv C_8$ is a constant, where we have used Lemma \ref{lemma2.1} and $d^*\theta=0$ above. By \eqref{3.48++} and $a \not\equiv  0$, we have $C_8=0$, i.e., $v=u_\infty$. Combining with \eqref{3.47--} yields
\begin{align}\label{3.51}
u(\cdot,t_{k_j})\rightarrow u_\infty \text{\ in\ } C^{1,\alpha}(M),
\end{align}
as $j\rightarrow +\infty$, which contradicts \eqref{3.46--}. Therefore, 
\begin{align}
\lim_{t\rightarrow +\infty}\|u(\cdot, t)-u_\infty(\cdot)\|_{C^{1, \alpha}(M)}=0.
\end{align}
\qed

\section*{Statements and Declarations}
\textbf{Data availability.} No data was used for the research described in the article.

\textbf{Conflict of interest.}  The author states that there is no conflict of interest. The author has no relevant financial or non-financial interests to disclose.

\bigskip
\bigskip

Weike Yu

School of Mathematical Sciences,

Ministry of Education Key Laboratory of NSLSCS,

Nanjing Normal University,

Nanjing, 210023, Jiangsu, P. R. China,

wkyu2018@outlook.com

\bigskip

\end{document}